\documentclass{article}
\usepackage{amsmath, amsthm, amsfonts, amssymb}
\usepackage{hyperref}
\newtheorem{theorem}{Theorem}[section]

\newtheorem{proposition}[theorem]{Proposition}
\newtheorem{corollary}[theorem]{Corollary}
\theoremstyle{definition}
\newtheorem{definition}[theorem]{Definition}

\theoremstyle{remark}
\newtheorem{remark}[theorem]{Remark}
\numberwithin{equation}{section}

\begin{document}
\title{\bf 2-Normed $\mathbb D$-Modules and Hahn-Banach Theorem for $\mathbb D$-linear 2-functionals}
\date{\textbf{Kulbir Singh and Romesh Kumar}}
\vspace{0in}
\maketitle
$\textbf{Abstract.}$ In this paper we introduce the notion of $\mathbb D$-valued 2-norm on  hyperbolic or $\mathbb D$-valued modules. Further, we define $\mathbb D$-linear 2-functional on these modules and consider some of their properties. We also establish the Hahn-Banach type extension theorem for $\mathbb D$-linear 2-functionals. \\\\
$\textbf{Keywords.}$ Hyperbolic modules, hyperbolic-valued norm, real-valued 2-norm, linear 2-functionals, Hahn-Banach theorems.\\\\
$\textbf{AMS  Subject Classification.}$ 46A70, 46A22.

\begin{section} {Introduction}

\renewcommand{\thefootnote}{\fnsymbol{footnote}}
\footnotetext{ The research of K. Singh is supported by CSIR-UGC (New-Delhi, India).}

The notion of linear 2-normed spaces was initially introduced by S. Gahler \cite{GHL}. Since then, many researchers have studied these spaces from different points of view and obtained various results, see for instance \cite{Cho, CD 1, CD 2, RE, LAL, AW}. In \cite{Zofia 1, Zofia 2}, Lewandowska gave a generalization of the Gahler's 2-normed space. The notion of 2-normed spaces is basically a two dimensional analogue of a normed space which got more attention after the publication of a paper \cite{AW}. In this paper, A. White defined and investigated the concept of bounded linear 2-functionals from $X\rightarrow X$, where $X$ denotes a 2-normed real linear space. Further, he proved a Hahn-Banach type extension theorem for linear 2-functionals on 2-normed real linear spaces. Later, S. N. Lal et. al in \cite{LAL} introduced 2-normed complex linear spaces and established a Hahn-Banach extension theorem for complex linear 2-functionals. For Hahn-Banach theorems for normed modules, one can refer to \cite{LL} and \cite{Hahn}.  \\\\
 In the present paper, we inroduce the notion of 2-normed $\mathbb D$-modules over the commutating non-division ring $\mathbb D$ of hyperbolic numbers and prove the Hahn-Banach theorem for $\mathbb D$-linear 2-functionals. Section 2, concentrates on some basic facts about hyperbolic numbers and $\mathbb D$-valued modules. In section 3, we define 2-normed $\mathbb D$-modules. Further, $\mathbb D$-linear 2-functionals on such modules and some of their properties are given in section 4. Finally, section 5 present Hahn-Banach theorems for 2-normed $\mathbb D$-modules.  
\end{section}
\begin{section} {A Review of Hyperbolic Numbers}
 In this section we summarize some basic properties of hyperbolic numbers which can be found in more details in \cite{YY, GG, KS2, Hahn, RR} and the references therein. The ring of hyperbolic numbers is the commutative ring $\mathbb D$ defined as 
$$ \mathbb D=\left\{ a+ \textbf{k} b \;\;|\;\; a, b \;\in \mathbb R,\; \textbf{k}^2=1 \;\;\text{with}\;\; \textbf{k} \notin \mathbb R\right\}.$$
Let $z=  a+ \textbf{k} b \in \mathbb D$. Then the $\dagger$-conjugation on $z$ is given by $$ z^\dagger=  a- \textbf{k} b.$$ 
This $\dagger$-conjugation on $\mathbb D$ is an additive, involutive and multiplicative in nature. A hyperbolic number $z= a+ \textbf{k} b$ is said to be an invertible if $zz^\dagger=a^2-b^2 \neq 0$. Thus, inverse of $z\in \mathbb D$ is given by $$ z^{-1}= \frac{z^\dagger}{zz^\dagger}\;.$$ If both $a$ and $b$ are non-zero but $a^2-b^2=0$, then $z$ is a zero-divisior in $\mathbb D$. We denote the set of all zero-divisiors in $\mathbb D$ by $\mathcal {NC}_\mathbb D$, that is, $$ \mathcal {NC}_\mathbb D= \left\{ z= a+ \textbf{k} b\;|\; z\neq 0,\; zz^\dagger= a^2-b^2=0\right\}.$$ 
The ring $\mathbb D$ of hyperbolic numbers is not a division ring as one can see that if $$e_1=\frac{1}{2} (1+\textbf{k})$$ and its $\dagger$-conjugate $$ e_1^\dagger =e_2=\frac{1}{2} (1-\textbf{k}),$$ then $e_1.e_2=0$, i.e., $e_1$ and $e_2$ are zero-divisiors in the ring $\mathbb D$. The numbers $e_1$ and $e_2$ are mutually complementary idempotent components. They make up the so called idempotent basis of hyperbolic numbers. Thus, every hyperbolic number $z= a+ \textbf{k} b$ in $\mathbb D$ can be written as :
\begin{align}  \label{hyp} z=e_1\alpha_1+ e_2 \alpha_2,\end{align}
where $\alpha_1= a+b$ and $\alpha_2= a-b$ are real numbers. Formula (\ref{hyp}) is called the idempotent representation of a hyperbolic number. Further, the two sets $e_1\mathbb D$ and $e_2\mathbb D$ are (principal) ideals in the ring $\mathbb D$ such that $e_1\mathbb D \cap e_2\mathbb D =\left\{0\right\}$ and $e_1\mathbb D + e_2\mathbb D= \mathbb D$. Hence, we can write \begin{align}\label{hyp 1} \mathbb D= e_1\mathbb D + e_2\mathbb D. \end{align} Formula (\ref{hyp 1}) is called the idempotent decomposition of $\mathbb D$. Thus the algebraic operations of addition, multiplication, taking of inverse, etc. can be realized component-wise. The set of non-negative hyperbolic numbers is given by (see \cite [P. 19] {YY}), 
$$ \mathbb D^+=\left\{ z= e_1\alpha_1+ e_2 \alpha_2\;|\; \alpha_1, \;\alpha_2 \geq 0\right\}.$$ Further, for any $z, u \in \mathbb D$, we write $ z\leq ' u$ whenever $u-z \in \mathbb D^+$ and it defines a partial order on $\mathbb D$. Also, if we take $z, u \in \mathbb R$, then $z\leq ' u$ if and only if $z\leq u$. Thus $\leq'$ is an extension of the total order $\leq$ on $\mathbb R$. 
If $A \subset \mathbb D$ is $\mathbb D$-bounded from above, then the $\mathbb D$-supremum of $A$ is defined as
$$\text{sup}_\mathbb D A = e_1\text{sup} A_1 + e_2 \text{sup}A_2 ,$$
where $A_1 = \left\{a_1 : e_1a_1 + e_2a_2 \in A\right\}$ and  $A_2 = \left\{a_2 : e_1a_1 + e_2a_2 \in A\right\}$.
Similarly, $\mathbb D$-infimum of a $\mathbb D$-bounded below set $A$ can be defined.\\
For any $z=e_1\alpha_1+e_2\alpha_2 \in \mathbb D$, the hyperbolic-valued modulus on $\mathbb D$ is given by   \begin{align} |z|_\textbf{k}= |e_1\alpha_1+e_2\alpha_2|_\textbf{k}= e_1|\alpha_1|+e_2|\alpha_2| \in \mathbb D^+, \end{align}  
where $|\alpha_1|$ and $|\alpha_2|$ denote the usual modulus of real numbers $\alpha_1$ and $\alpha_2$ respectively. For more details, see (\cite[Section 1.5] {YY}, \cite{Hahn} and \cite{RR}).\\
Let $X$ be a $\mathbb D$-module. Consider the sets
$$ X_1=e_1X \;\;\text{and}\;\; X_2=e_2X.$$
Then $$ X_1 \cap X_2= \left\{0\right\}\;\;\text{and}$$
 \begin{align} \label{PPP} X=e_1X_1+e_2X_2, \end{align}
 where $X_1$ and $X_2$ are real linear spaces as well as $\mathbb D$-modules. Formula (\ref{PPP}) is called the idempotent decomposition of $X$. Thus, any $x\in X$ can be uniquely written as $x=e_1x_1+e_2x_2$ with $x_1\in X_1$ and $x_2 \in X_2$.
Further, it can be shown that if $U$ and $W$ be any two real linear spaces, then $X=e_1 U+ e_2 W$ will be a $\mathbb D$-module. We denote the set of all zero-divisiors in $X$ by $\mathcal {NC}_X$, that is, $$ \mathcal {NC}_X= \left\{ 0 \neq z \in X\;:\; z\in e_1X\; \cup\;e_2X\right\}.$$  
\begin{definition} Let $X$ be a $\mathbb D$-module and $\|.\|_\mathbb D: X\rightarrow \mathbb D^+$ be a function such that for any $x, y\in X$ and $\alpha \in \mathbb D$, it satisfies the following properties:
 \begin{enumerate}
 \item[(i)] $\|x\|_\mathbb D =0 \Leftrightarrow x=0$.
 \item[(ii)] $\|\alpha x\|_\mathbb D= |\alpha|_\textbf{k} \|x\|_\mathbb D$.
 \item[(iii)] $\|x+y\|_\mathbb D \leq' \|x\|_\mathbb D + \|y\|_\mathbb D$.
 \end{enumerate}
 Then we say that $\|.\|_\mathbb D$ is a hyperbolic or $\mathbb D$-valued norm on $X$.\end{definition} 
 The hyperbolic-valued norm on hyperbolic modules has been intensively discussed in \cite{YY, Hahn} and many other references therein.
 \end{section}

\begin{section} {2-normed $\mathbb D$-modules}
In this section we define $\mathbb D$-valued 2-norm on hyperbolic or $\mathbb D$-valued modules. Further, we show that how this $\mathbb D$-valued 2-norm on a $\mathbb D$-module $X$ is related to the real 2-norm on the real idempotent components of $X$.
 
 \begin{definition} Let $X$ be a $\mathbb D$-module of dimension greater 1. A map $$\|.,.\|_\mathbb D:X \times X \rightarrow \mathbb D$$ is said to be $\mathbb D$-valued 2-norm on $X$ if for all $x,\;y,\;z \in X$ and $\alpha \in \mathbb D$ it satisfies the following properties:
 \begin{enumerate}
 \item[(i)]  $\|x,y\|_\mathbb D=0$ if and only if $x$, $y$ are linearly dependent,
 \item[(ii)] $\|x,y\|_\mathbb D=\|y,x\|_\mathbb D$,
 \item[(iii)] $\|\alpha x,y\|_\mathbb D =|\alpha|_\textbf{k}\|x,y\|_\mathbb D$,
 \item[(iv)] $\|x+y,z\|_\mathbb D \leq' \|x,z\|_\mathbb D+\|y,z\|_\mathbb D.$
 \end{enumerate}
 Then the pair $(X,\|.,.\|_\mathbb D)$ is called a 2-normed $\mathbb D$-module. Further, it can be shown that $\|x,y\|_\mathbb D \in \mathbb D^+$ and $\|x,y+\alpha x \|_\mathbb D= \|x,y\|_\mathbb D$ $\forall \;x,\;y \in X$ and $\forall \;\alpha \in \mathbb D$.  
 \end{definition} 
\begin{remark} Let $X_1$ and $X_2$ be two arbitrary real linear 2-normed spaces such that dim $(X_1)> 1$ and dim $(X_2)> 1$ with respective real 2-norms $\|.,.\|_1$ and  $\|.,.\|_2$. Let $X=e_1 X_1+e_2 X_2$. Then $X$ will form a $\mathbb D$-module with dim $(X)> 1$.\\\\ For $x=e_1x_1+e_2x_2$, $y=e_1y_1+e_2y_2 \in X$, we define
 \begin{align}\label{2-n} \|x,y\|_\mathbb D= \|e_1x_1+e_2x_2,e_1y_1+e_2y_2\|_\mathbb D=e_1\|x_1,y_1\|_1+e_2\|x_2,y_2\|_2.\end{align}
 Then the Formula (\ref{2-n}) is a $\mathbb D$-valued 2-norm on $X$ can be verified easily as follows:
  \begin{align*}\;\|x,y\|_\mathbb D=0&\Leftrightarrow e_1\|x_1,y_1\|_1+e_2\|x_2,y_2\|_2=0\\
 &\Leftrightarrow \|x_1,y_1\|_1=0 \;\text{and}\; \|x_2,y_2\|_2=0\\
 &\Leftrightarrow x_1, y_1 \text{ are linearly dependent and} \;x_2, y_2 \text{ are linearly dependent}\\
 &\Leftrightarrow  x \;\text{and}\; y \;\text{are linearly dependent}.
 \end{align*}
 Clearly,\;\; $\|x,y\|_\mathbb D=\|y,x\|_\mathbb D$.\\\\
 $\text{Further,\; for\; any}\; \alpha \in \mathbb D$,
\begin{align*}  \|\alpha x,y\|_\mathbb D&=\|(e_1\alpha_1+e_2\alpha_2)(e_1x_1+e_2x_2),e_1y_1+e_2y_2\|_\mathbb D\\
 &=\|e_1(\alpha_1 x_1)+e_2(\alpha_2 x_2),e_1y_1+e_2y_2\|_\mathbb D\\
 &=e_1\|\alpha_1 x_1,y_1\|_1+ e_2\|\alpha_2 x_2,y_2\|_2\\
 &=e_1|\alpha_1|\|x_1,y_1\|_1+e_2|\alpha_2|\|x_2,y_2\|_2\\
 &=(e_1 |\alpha_1|+e_2 |\alpha_2|)\;( e_1\|x_1,y_1\|_1+e_2\|x_2,y_2\|_2)\\
 &=|\alpha|_\textbf{k} \|x,y\|_\mathbb D.
 \end{align*}
 Finally, let $x,\;y,\;z \in X$. Then
\begin{align*} 		  \|x+y,z\|_\mathbb D&=\|(e_1x_1+e_2x_2)+(e_1y_1+e_2y_2), (e_1z_1+e_2z_2)\|_\mathbb D\\
 &=\|e_1(x_1+y_1)+e_2(x_2+y_2), e_1z_1+e_2z_2\|_\mathbb D\\
 &=e_1\|(x_1+y_1),z_1\|_1+e_2\|(x_2+y_2),z_2\|_2\\
 &\leq' e_1(\|x_1,z_1\|_1)+(\|y_1,z_1\|_1)+ e_2(\|x_2,z_2\|_2)+(\|y_2,z_2\|_2)\\
 &=(e_1 \|x_1,z_1\|_1+e_2 \|x_2,z_2\|_2)+ (e_1\|y_1,z_1\|_1 +e_2 \|y_2,z_2\|_2)\\
 &= \|x,z\|_\mathbb D+\|y,z\|_\mathbb D.
 \end{align*}  
 \end{remark} 
 \begin{proposition} \label{DD1} Let $X$ be a 2-normed $\mathbb D$-module. Then $e_1 X$ and $e_2 X$ can be seen as 2-normed real linear spaces with their norms induced by the $\mathbb D$-valued 2-norm on $X$.
 \end{proposition}
\begin{proof} Let $\|.,.\|_\mathbb D: X \times X \rightarrow \mathbb D$ be the $\mathbb D$-valued 2-norm on $X$. Then we can write it as
$$ \|x,y\|_\mathbb D= e_1\Phi(x,y) +e_2 \Psi(x,y), \; \forall\; (x,y)\in X \times X $$ where $\Phi, \Psi: X \times X \rightarrow \mathbb R$ are real-valued functions. For each $x, y \in X$, we have   
\begin{align*}  e_1\Phi(e_1x, e_1y) + e_2 \Psi(e_1x, e_1y)&= \|e_1x, e_1y\|_\mathbb D=e_1\|e_1x, e_1y\|_\mathbb D\\
&=e_1\left(e_1\Phi(e_1x, e_1y) + e_2 \Psi(e_1x, e_1y)\right)\\
&= e_1\Phi(e_1x, e_1y).
\end{align*}
This implies that $\Psi(e_1x, e_1y)=0\;\text{and}\; \|e_1x, e_1y\|_\mathbb D=e_1\Phi(e_1x, e_1y).$ Similarly, one can prove that $\Phi(e_2x, e_2y)=0\; \text{and}\; \|e_2x, e_2y\|_\mathbb D=e_2\Psi(e_1x, e_1y).$ Hence 
\begin{align}  \label{DD} \|x,y\|_\mathbb D &= e_1 \|x,y\|_\mathbb D+ e_2 \|x,y\|_\mathbb D=e_1 \|e_1x,y\|_\mathbb D+ e_2 \|e_2x,y\|_\mathbb D \notag\\
&=e_1\|y,e_1x\|_\mathbb D+ e_2 \|y,e_2x\|_\mathbb D=\|e_1y,e_1x\|_\mathbb D+ \|e_2y,e_2x\|_\mathbb D \notag\\
&= \|e_1x, e_1y\|_\mathbb D + \|e_2x, e_2y\|_\mathbb D \notag\\
&= e_1\Phi(e_1x, e_1y)+e_2 \Psi(e_2x, e_2y).
\end{align}
Since for any $x,y\in X$, $$e_1\Phi(e_1x, e_1y)=\|e_1x, e_1y\|_\mathbb D=\|e_1y, e_1x\|_\mathbb D= e_1\Phi(e_1y, e_1x).$$ Thus, $~~~~~~~~~~~~~~ \Phi(e_1x, e_1y)= \Phi(e_1y, e_1x).$\\ Similarly, $~~~~~~~~~ \Psi(e_2x, e_2y)= \Psi(e_2y, e_2x).$\\
Now for any $\lambda \in \mathbb R$ and $x,\;y\in X$, 
$$ \|\lambda x,y\|_\mathbb D= |\lambda|_\textbf{k} \|x,y\|_\mathbb D= |\lambda| \|x,y\|_\mathbb D. $$ 
This implies $$e_1\Phi(\lambda e_1x, e_1y)+e_2 \Psi(\lambda e_2x, e_2y)= |\lambda| e_1\Phi( e_1x, e_1y)+ |\lambda| e_2 \Psi(e_2x, e_2y) $$ and thus we have 
$~~~~~~~~~\Phi(\lambda e_1x, e_1y)= |\lambda| \Phi( e_1x, e_1y)$\;\; and $$\Psi(\lambda e_2x, e_2y)= |\lambda| \Psi(e_2x, e_2y).$$ For the triangular inequality, let $x, y, z \in X$ such that $$\|x+y,z\|_\mathbb D \leq' \|x,z\|_\mathbb D+\|y,z\|_\mathbb D.$$
Then by using (\ref{DD}), we have 
 $$e_1\Phi(e_1x+e_1y, e_1z)+e_2 \Psi(e_2x+e_2y, e_2z)\leq' $$ $$e_1\Phi(e_1x,e_1z)+ e_2 \Psi(e_2x, e_2z) + e_1\Phi(e_1y,e_1z)+ e_2 \Psi(e_2y, e_2z).$$
Hence  $$\Phi(e_1x+e_1y, e_1z) \leq \Phi(e_1x,e_1z) + \Phi(e_1y,e_1z)$$
and $$ \Psi(e_2x+e_2y, e_2z) \leq \Psi(e_2x, e_2z)+ \Psi(e_2y, e_2z).$$
Finally it remains to show that $\Phi(e_1x, e_1y)=0$ if and only if $e_1x$, $e_1y$ are linearly dependent and similarly for $\Psi$.
Firstly suppose that $\Phi(e_1x, e_1y)=0$. This means $e_1 \Phi(e_1x, e_1y)=0$ and hence $\|e_1x, e_1y\|_\mathbb D=0$. Since $\|.,.\|_\mathbb D$ is a 2-norm on $X$ implies that $e_1x$ and $e_1y$ are linearly dependent.\\
Conversly, suppose $x, y \in X$ such that $e_1x$ and $e_1y$ are linearly dependent. Then 
\begin{align*} \|x, y\|_\mathbb D &= e_1\|x, y\|_\mathbb D+e_2\|x, y\|_\mathbb D= \|e_1x, e_1y\|_\mathbb D + \|e_2x, e_2y\|_\mathbb D\\
&=\|e_2x, e_2y\|_\mathbb D = e_2 \Psi(e_2x, e_2y).
\end{align*}
Thus, by using (\ref{DD}), we have $\Phi(e_1x, e_1y)=0$ and similarly for $\Psi$. 
Hence $\Phi$ is a real 2-norm on the real linear space $e_1 X$ and $\Psi$ is a real 2-norm on the real linear space $e_2 X$.  
 \end{proof}
 \begin{definition} A sequence $\left\{x_n\right\}_{n\in \mathbb N}$ in a 2-normed $\mathbb D$-module $X$ is called a convergent sequence with respect to the $\mathbb D$-valued 2-norm $\|.,.\|_\mathbb D$ if there is an $x_0 \in \mathcal A$ such that $$\underset{n\to\infty}{\mathop{\lim}} \;\|x_n-x_0,y\|_\mathbb D=0\;, \; \text{for all}\; y\in X.$$
\end{definition} 										

\begin{remark} By using Proposition \ref{DD1}, we can write $\|.,.\|_\mathbb D=e_1 \|.,.\|_1+e_2 \|.,.\|_2$, where $\|.,.\|_l$ is a 2-norm on the real linear space $e_l\mathcal A$, for $l=1,2$. Then
\begin{align*} &\underset{n\to\infty}{\mathop{\lim}} \;\|x_n-x_0,y\|_\mathbb D=0\;, \; \forall\; y\in X \;\text{implies}\\
&\underset{n\to\infty}{\mathop{\lim}} \;\|e_lx_n-e_lx_0,e_ly\|_l=0, \;\forall\;\; e_l y \in e_l X, \; \text{for}\; l=1,2.
\end{align*}
This implies that the sequence $\left\{e_1x_n\right\}_{n\in \mathbb N}$ converges to $e_1x_0$ in $e_1 X$ and sequence $\left\{e_2x_n\right\}_{n\in \mathbb N}$ converges to $e_2x_0$ in $e_2 X$.
\end{remark}
\end{section}

\begin{section} 
{$\mathbb D$-Linear 2-Functional on 2-normed $\mathbb D$-modules}  
In this section we define $\mathbb D$-valued linear 2-functionals on 2-normed $\mathbb D$-modules and discuss some of their properties.
Let $X$ be a $\mathbb D$-module and let $\mathcal M$ and $\mathcal N$ be two submodules of $X$. A map $f: \mathcal M\times \mathcal N \rightarrow \mathbb D$ is called a $\mathbb D$-valued 2-functional on $\mathcal M\times \mathcal N$.  
\begin{definition} Let $f$ be a $\mathbb D$-valued 2-functional on $\mathcal M\times \mathcal N$. If $f$ is such that for each $\alpha, \beta \in \mathbb D$ and for all $x, y  \in \mathcal M$ and $z,w \in \mathcal N$ we have:
\begin{enumerate}
 \item[(i)]  $f(x+y, z+w)=f(x,z)+ f(y,z)+ f(x,w) + f(y,w)$,
 \item[(ii)] $f(\alpha x, \beta z)= \alpha \beta f(x,z)$,
\end{enumerate}
then $f$ is called a $\mathbb D$-linear 2-functional with domain $\mathcal M\times \mathcal N$. Further, it is easy to show that if $x$ and $y$ are linear dependent in $X$, then $f(x,y)=0$ for $(x,y) \in \mathcal M\times \mathcal N$.\\
Let $f: \mathcal M\times \mathcal N \rightarrow \mathbb D$  be a $\mathbb D$-linear 2-functional. For any $x, z\in \mathcal M\times \mathcal N$, one can write \begin{align} f(x,z)= \phi(x,z)+ \textbf{k} \psi(x,z)=e_1 f_1(x,z)+e_2f_2(x,z),\end{align} where $\phi(x,z) \in \mathbb R$, $\psi(x,z) \in \mathbb R$, $f_1(x,z) \in \mathbb R$ and $f_2(x,z) \in \mathbb R$ with $f_1= \phi+\psi$ and $f_2=  \phi-\psi$.\\\\ Let us first show that $f_1, f_2: \mathcal M_\mathbb R\times \mathcal N_\mathbb R\rightarrow \mathbb R$ are real linear 2-functionals, where $\mathcal M_\mathbb R$ and $\mathcal N_\mathbb R$ are real linear subspaces of $X$.\\      
Given $\alpha=\alpha_1+\textbf{k}\alpha_2$, $\beta=\beta_1+\textbf{k}\beta_2 \in \mathbb D$ with $x, y  \in \mathcal M$ and $z,w \in \mathcal N$, one has:
\begin{align*} &f(x+y, z+w)=  e_1 f_1(x+y,z+w)+e_2f_2(x+y,z+w)\;\; \text{implies}\\
&	f(x,z)+ f(y,z)+ f(x,w) + f(y,w)= e_1 f_1(x+y,z+w)+e_2f_2(x+y,z+w).
\end{align*}
Then\;\; $e_1 (f_1(x,z)+f_1(y,z)+ f_1(x,w)+ f_1(y,w))+ e_2(f_2(x,z) + f_2(y,z)$\\ $~~~~~~~~~~~~+ f_2(x,w)+f_2(y,w))=  e_1 f_1(x+y,z+w)+e_2f_2(x+y,z+w).$\\\\
Hence\;\; 
$f_1(x,z)+f_1(y,z)+ f_1(x,w)+ f_1(y,w)=f_1(x+y,z+w)$ and 
 $$ f_2(x,z) + f_2(y,z)+ f_2(x,w)+f_2(y,w)=f_2(x+y,z+w).$$
Further,  $f(\alpha x, \beta z )= e_1 f_1(\alpha x, \beta z)+e_2f_2(\alpha x, \beta z)$\;\; implies\\
 $~~~~~~~~~~~~~\alpha \beta f( x,z)=  e_1 f_1(\alpha x, \beta z)+e_2f_2(\alpha x, \beta z).$\\\\ That is,\;\;
 $e_1((\alpha_1+\alpha_2) (\beta_1+\beta_2) f_1(x,z))+ e_2((\alpha_1-\alpha_2) (\beta_1-\beta_2) f_2(x,z))$\\$~~~~~~~~~~~= e_1 f_1((\alpha_1+\textbf{k}\alpha_2) x, (\beta_1+\textbf{k}\beta_2) z)+e_2f_2( (\alpha_1+\textbf{k}\alpha_2)x, (\beta_1+\textbf{k}\beta_2) z)$.\\\\ This implies $(\alpha_1+\alpha_2) (\beta_1+\beta_2) f_1(x,z)=f_1((\alpha_1+\textbf{k}\;\alpha_2) x, (\beta_1+\textbf{k}\;\beta_2) z)$ and
$$(\alpha_1-\alpha_2) (\beta_1-\beta_2) f_2(x,z)=f_2((\alpha_1+\textbf{k}\;\alpha_2) x, (\beta_1+\textbf{k}\;\beta_2) z).$$ In particular, setting $\alpha_2=0$ and $\beta_2=0$, we have:
$$ \alpha_1 \beta_1 f_1(x,z)=f_1(\alpha_1 x, \beta_1z)\; \;\text{and}\; \;\alpha_1 \beta_1 f_2(x,z)=f_2(\alpha_1 x, \beta_1z).$$ Thus, the mappings $f_1$ and $f_2$ are real linear 2-functionals on $\mathcal M_\mathbb R\times \mathcal N_\mathbb R$. 
Similarly one can show that $\phi$ and $\psi$ are also real linear 2-functionals on $\mathcal M_\mathbb R\times \mathcal N_\mathbb R$. 
\end{definition}
\begin{remark} \label{BDD} Let $X$ be a $\mathbb D$-module and $\mathcal M$ and $\mathcal N$ be any two submodules of $X$. Then we can write $\mathcal M=e_1\mathcal M_1+e_2\mathcal M_2$ and $\mathcal N=e_1\mathcal N_1+e_2\mathcal N_2$, where $\mathcal M_l= e_l \mathcal M$ and $\mathcal N_l= e_l \mathcal N$, for $l=1,2$, can be seen as real linear subspaces of $\mathcal M$ and $\mathcal N$ respectively. Let $f$ be a $\mathbb D$-linear 2-functional on $\mathcal M\times \mathcal N$. Then for every $x=e_1x_1+e_2x_2 \in \mathcal M$ and $y=e_1y_1+e_2y_2 \in \mathcal N$, \begin{align}& f(x,y)=f(e_1x_1+e_2x_2,e_1y_1+e_2y_2)\notag\\
&=f(e_1x_1,e_1y_1)+f(e_1x_1,e_2y_2)+f(e_2x_2,e_1y_1)+f(e_2x_2,e_2y_2) \notag\\
&=e_1f(e_1x_1,e_1y_1)+e_2f(e_2x_2,e_2y_2).
\end{align} Thus, by using (4.1) and (4.2), we have \begin{align*}  & f(e_1x_1,e_1y_1)= e_1 f_1(e_1x_1,e_1y_1)+e_2f_2(e_1x_1,e_1y_1).\end{align*}
This implies  $e_1f(e_1x_1, e_1y_1)= e_1 f_1(e_1x_1,e_1y_1)\; \text{and}\; f_2(e_2x_1,e_2y_1)=0$. That is, $ f(e_1x_1, e_1y_1)=  f_1(e_1x_1,e_1y_1).$ Similarly, $f(e_2x_2, e_2y_2)=  f_2(e_2x_2, e_2y_2).$  
Hence, we can write  \begin{align} f(x,y) =e_1f_1(e_1x_1,e_1y_1)+e_2f_2(e_2x_2,e_2y_2),\end{align}

where $f_1: \mathcal M_1\times \mathcal N_1\rightarrow \mathbb R$ and $f_2: \mathcal M_2\times \mathcal N_2\rightarrow \mathbb R$ are real linear 2-functionals.
\end{remark}
\begin{remark} Let $f$ be a $\mathbb D$-linear 2-functional with domain $\mathcal M\times \mathcal N$. For each $(x,y) \in \mathcal M\times \mathcal N$, we can write 
\begin{align} f(x,y)= \phi(x,y)+ \textbf{k} \psi(x,y), \end{align}
where $\phi, \psi$ are functions from $\mathcal M\times \mathcal N$ to $\mathbb R$. Moreover, $\phi, \psi$ are real linear 2-functionals. Further, $\phi(\textbf{k}x,y)+ \textbf{k} \psi(\textbf{k}x,y)= f(\textbf{k}x, y)= \textbf{k} f(x,y)= \textbf{k} (\phi(x,y)+ \textbf{k} \psi(x,y))= \psi(x,y)+ \textbf{k} \phi(x,y)$.\\ Then\; $\phi(\textbf{k} x,y)= \psi(x,y)$ and $\psi(\textbf{k}x,y)= \phi(x,y)$. Thus, (4.4) becomes \begin{align} f(x,y)= \phi(x,y)+ \textbf{k} \phi(\textbf{k}x,y).\end{align} 
Similarly, we can also define \begin{align} f(x,y)= \phi(x,y)+ \textbf{k} \phi(x,\textbf{k}y).\end{align} 
Thus, for each $\mathbb D$-linear 2-functional $f$, there is a  real linear 2-functional $\phi$ associated to $f$ by the formula (4.5) and (4.6). However, if $\phi$ be any real linear 2-functional and if we define $f$ as in (4.5) or (4.6), then it can be seen that $f$ need not be a $\mathbb D$-linear 2-functional.    
\end{remark}
\begin{definition} \label{DSE} Let $X$ be a 2-normed $\mathbb D$-module and let $\mathcal M$ and $\mathcal N$ be two submodules of $X$. A linear 2-functional $f:\mathcal M\times \mathcal N\rightarrow \mathbb D$ is called $\mathbb D$-bounded if there exists $\Delta \in \mathbb D^+$ such that $$ |f(x,y)|_\textbf{k} \leq^{'} \Delta. \|(x,y)\|_\mathbb D, \; \forall\;(x,y) \in \mathcal M\times \mathcal N.$$
If $f$ is $\mathbb D$-bounded, then the hyperbolic norm of $f$ is given by 
\begin{align}\label{llp} \|f\|_\mathbb D= \text{inf}_\mathbb D \left\{ \Delta\;:\;|f(x,y)|_\textbf{k} \leq^{'} \Delta. \|(x,y)\|_\mathbb D\right\}, \; \forall\;(x,y) \in \mathcal M\times \mathcal N.\end{align} \end{definition}
Setting $f= e_1f_1+e_2f_2$ and $\Delta=e_1\Delta_1+e_2\Delta_2$ with $\Delta_1,\;\Delta_2 \in \mathbb R^+$
and by using Remark \ref{BDD}, one gets: 
\begin{align*} 
\|f\|_\mathbb D&=\|e_1f_1+e_2f_2\|_\mathbb D\\
&= \text{inf}_\mathbb D \left\{ \Delta\;:\;|f(x,y)|_\textbf{k} \leq^{'} \Delta. \|(x,y)\|_\mathbb D\right\}\\
&=\text{inf}_\mathbb D \left\{ e_1\Delta_1+e_2\Delta_2\right\}\;\; \text{such that}
\end{align*} 
$$|e_1f_1(e_1x_1,e_1y_1)+e_2f_2(e_2x_2,e_2y_2)|_\textbf{k}  \leq^{'} (e_1\Delta_1+e_2\Delta_2)(e_1\|x_1,y_1\|_1+e_2\|x_2,y_2\|_2).$$ 
 This implies $\|f\|_\mathbb D= e_1 \text{inf}\left\{\Delta_1\right\}+ e_2 \text{inf}\left\{\Delta_2\right\}$ such that
$$|f_1(e_1x_1,e_1y_1)| \leq \Delta_1 \; \|e_1x_1,e_1y_1\|_1 \;\;\text{and} \;\;|f_2(e_2x_2,e_2y_2)| \leq \Delta_2 \; \|e_2x_2,e_2y_2\|_2.$$
Thus, the hyperbolic norm of $f$ can also be defined as   \begin{align} \|f\|_\mathbb D=e_1\|f_1\|_1+e_2\|f_2\|_2,\end{align}
where $\|f_l\|_1= \text{inf}\left\{\Delta_l \;:\; |f_l(e_lx_l,e_ly_l)| \leq \Delta_l \; \|e_lx_l,e_ly_l\|_l\right\}$,\; $l=1,2$. 
\begin{remark} \label{nnn}                                  
In real 2-normed spaces, the norm of bounded real linear 2-functional $f_1$ on $\mathcal M_1 \times \mathcal N_1$ can be defined as 
\begin{align*} \|f_1\|_1&=\text{sup} \left\{\frac{|f_1(x_1,y_1)|}{\|x_1,y_1\|_1}\;:\;\|x_1,y_1\|_1 \neq 0,\; (x_1,y_1)\in \mathcal M_1 \times \mathcal N_1\right\}\\
&=\text{sup} \left\{|f_1(x_1,y_1)|\;:\;\|x_1,y_1\|_1=1, \;(x_1,y_1)\in \mathcal M_1\times \mathcal N_1\right\}  
\end{align*} and similarly for $f_2$. 
Thus, by using these equalities, formula (4.8) gives
\begin{align*} \|f\|_\mathbb D&= e_1\text{sup} \left\{\frac{|f_1(x_1,y_1)|}{\|x_1,y_1\|_1}\;:\;\|x_1,y_1\|_1 \neq 0,\; (x_1,y_1)\in \mathcal M_1 \times \mathcal N_1\right\}\\
&+ e_2\text{sup} \left\{\frac{|f_2(x_2,y_2)|}{\|x_2,y_2\|_2}\;:\;\|x_2,y_2\|_2 \neq 0,\: (x_2,y_2)\in \mathcal M_2 \times \mathcal N_2\right\}\\
&= \text{sup}_\mathbb D \left\{\frac{e_1|f_1(x_1,y_1)|+e_2 |f_2(x_2,y_2)|}{e_1\|x_1,y_1\|_1+e_2 \|x_2,y_2\|_2}\right\}\;\text{such\; that}\\
&  ~~~~~~~~~~~~\;(e_1\|x_1,y_1\|_1+e_2 \|x_2,y_2\|_2 \notin \mathcal {NC}_\mathbb D \cup \left\{0\right\})\\
&=  \text{sup}_\mathbb D \left\{\frac{|e_1f_1(x_1,y_1)+e_2 f_2(x_2,y_2)|_\textbf{k}}{\|e_1(x_1,y_1)+e_2 (x_2,y_2)\|_\mathbb D}\right\}\;\text{such\; that}\\
& ~~~~~~~~~~~~~(\|e_1(x_1,y_1)+e_2 (x_2,y_2)\|_\mathbb D \notin \mathcal {NC}_\mathbb D \cup \left\{0\right\})\\
&=\text{sup}_\mathbb D \left\{\frac{|f(x,y)|_\textbf{k}}{\|x,y\|_\mathbb D}\;:\;\|x,y\|_\mathbb D \notin \mathcal {NC}_\mathbb D \cup \left\{0\right\}, \text{where}\;(x,y)\in \mathcal M\times \mathcal N\right\}.
\end{align*}
Similarly, one can also prove that \begin{align*} \|f\|_\mathbb D&= \text{sup}_\mathbb D \left\{|f(x,y)|_\textbf{k}\;;\;\|x,y\|_\mathbb D=1,\;\;(x,y)\in \mathcal M\times \mathcal N\right\}.
\end{align*}
\end{remark}
 \end{section}
\begin{section} 
{The Hahn-Banach theorem for $\mathbb D$-linear\\ 2-functionals}
In this section we prove the  Hahn-Banach theorem for 2-normed $\mathbb D$-modules. Further, we discuss one of its important corollary.
  
\begin{theorem} \label{2 hahn} Let $X$ be a 2-normed $\mathbb D$-module and $z \in X$. Let $[z]$ denotes the $\mathbb D$-submodule of $X$ generated by $z$ and $\mathcal M$ be any $\mathbb D$-submodule of $X$. Let $f$ be a $\mathbb D$-bounded linear 2-functional on $\mathcal M \times [z]$ $(or \;on \;[z] \times \mathcal M)$. Then there exists a   $\mathbb D$-bounded linear 2-functional $F$ on $X \times [z]$ $\left(or \;on \;[z] \times X \right) $ such that 
$$ \|F\|_\mathbb D = \|f\|_\mathbb D\; \text{and}\;\;$$
$$ F(x,\alpha z)= f(x,\alpha z), \;\;\text{for\;all}\;\;\; (x,\alpha z)\in \mathcal M \times [z]$$
 $$(\text{or}\;\;F(\alpha z,x)= f(\alpha z,x), \;\;\text{for\;all}\;\;\; (\alpha z,x)\in [z] \times \mathcal M).$$
\end{theorem}
\begin{proof} Let $f$ be defined on $\mathcal M \times [z]$ be a $\mathbb D$-bounded linear 2-functional. Note that if we take $[z] \times \mathcal M$ as domain of $f$, the proof will follows on the similar lines.\\ 
\textbf{Case 1}: Suppose $z \notin \left\{\mathcal {NC}_X\cup \left\{0\right\}\right\}$.
If $\mathcal M \neq X $, then there exists $ x' \in X - \mathcal M$. Define 
$$ \mathcal N= \left\{ x+ \beta x'\;|\; x\in \mathcal M, \;\beta \in \mathbb D \right\}= \mathcal M + \left\{ \beta x' \;|\;\beta \in \mathbb D \right\}.$$   
Clearly $\mathcal N$ is a $\mathbb D$-module. Let $x, y \in \mathcal M$. Then 
\begin{align*} f( x, z)-f(y, z)&= f(x-y,z) \leq \|f\|_\mathbb D \|x-y,z\|_\mathbb D\\
&= \|f\|_\mathbb D \|(x+x')-(y+x'),z\|_\mathbb D\\
&\leq'  \|f\|_\mathbb D \|x+x',z\|_\mathbb D +  \|f\|_\mathbb D \|y+x',z\|_\mathbb D.
\end{align*}
Thus, \begin{align} -\|f\|_\mathbb D \|y+x',z\|_\mathbb D-f(y, z) \leq' \|f\|_\mathbb D \|x+x',z\|_\mathbb D-f( x, z). \end{align}
Therefore, 
\begin{align*} m_0&= \text{sup}_{y \in \mathcal M} \left\{-\|f\|_\mathbb D \|y+x',z\|_\mathbb D-f(y, z)\right\}\\
&\leq' \text{inf}_{x \in \mathcal M} \left\{\|f\|_\mathbb D \|x+x',z\|_\mathbb D-f( x, z)\right\}= m.
\end{align*}
Choose $r \in \mathbb D$ such that $ m_0 \leq' r \leq' m$. Setting y=x in (5.1), we get
$$ -\|f\|_\mathbb D \|x+x',z\|_\mathbb D-f(x, z) \leq'r \leq' \|f\|_\mathbb D \|x+x',z\|_\mathbb D-f( x, z).$$  
 \begin{align} \text{That\; is},\;~~~| f(x,z)+r|_\textbf{k} \leq' \|f\|_\mathbb D \|x+x',z\|_\mathbb D. \end{align}
For any $(x+\beta x', \alpha z) \in \mathcal N \times [z]$, define $g$ on $\mathcal N \times [z]$ as
  $$g(x+\beta x', \alpha z)= \alpha f(x,z) + \alpha \beta r.$$ 
Clearly, $g$ is a $\mathbb D$-extension of $f$. Further, it is easy to see that $g$ is $\mathbb D$-linear 2-functional. To show that $g$ is $\mathbb D$-bounded, replace $x$ by $x / \beta$ in (5.2), where $ \beta \notin \mathcal {NC}_\mathbb D$ is a non zero hyperbolic number. Then we get 
\begin{align*} | f(x,z)+ \beta r|_\textbf{k} \leq' \|f\|_\mathbb D \|x+ \beta x',z\|_\mathbb D. \end{align*} 
Thus, \begin{align*} |g(x+\beta x', \alpha z)|_\textbf{k} &= | \alpha f(x,z) + \alpha \beta r|_\textbf{k} = |\alpha|_\textbf{k} |f(x,z) + \beta r|_\textbf{k}\\
&\leq' |\alpha|_\textbf{k} \|f\|_\mathbb D \|x+ \beta x',z\|_\mathbb D\\
&= \|f\|_\mathbb D \|x+ \beta x', \alpha z\|_\mathbb D.
\end{align*}
 Hence $g$ is $\mathbb D$-bounded and $\|g\|_\mathbb D \leq' \|f\|_\mathbb D.$  Thus $\|g\|_\mathbb D = \|f\|_\mathbb D.$\\\\
Now consider the family  $\mathcal P$ of all pairs $(\mathcal M', f'),$ where $\mathcal M'$ is a $\mathbb D$-submodule of $X$ such that $\mathcal M \subseteq \mathcal M'$ and $f'$ is $\mathbb D$-bounded linear 2-functional on $ \mathcal M' \times [z]$ with $f' |_\mathcal M = f$ and $\|f'\|_\mathbb D = \|f\|_\mathbb D.$ A partial order $\prec$ on $\mathcal P$ is introduced as follows:  $(\mathcal M', f') \prec (\mathcal M'', f'')$ if and only if $\mathcal M' \subseteq \mathcal M''$ and $f''$ is an extension of $f'$ with  $\|f''\|_\mathbb D = \|f'\|_\mathbb D.$ Let $\mathcal T$ be any linearly ordered subset of $\mathcal P$. Let  $\widetilde{\mathcal M}=\bigcup \mathcal M'$ such that $(\mathcal M', f') \in \mathcal T$. Clearly $\widetilde{\mathcal M}$ is a $\mathbb D$-linear submodule of $X$ containing $\mathcal M$. For each $x\in \widetilde{\mathcal M}$ there exists $\mathcal M'\in \mathcal T$ with $x\in \mathcal M'$. Define  $\widetilde{f}: \widetilde{\mathcal M} \times [z] \rightarrow \mathbb D$ as $ \widetilde{f}(x, \alpha z)= h(x, \alpha z),$ where $h$ is associated with some $\mathcal N$ such that $(\mathcal N, h) \in \mathcal T$, which contains $x$. Then $\widetilde{f}$ is well defined as $\mathcal T$ is linearly ordered set. Thus the constructed pair $(\widetilde{\mathcal M}, \widetilde{f})$ is hence an upper bound for the linearly ordered set $\mathcal T$. Hence by Zorn's Lemma, $\mathcal P$ contains a maximal element $(\mathcal A, F)$. Further, $\mathcal A=X$, for if not, then there exists $(\widehat{\mathcal M}, \widehat{f}) \in \mathcal P$ such that $(\mathcal A, F) \prec (\widehat{\mathcal M}, \widehat{f})$, which contradicts the maximality of $(\mathcal A, F)$. Thus, for Case 1, the theorem is established.\\

\textbf{Case 2}: Now let $z \in \mathcal {NC}_X$ be a non zero element. Then either $z=e_1z$ or $z=e_2z$. Suppose $z=e_1z$. We can choose some $z' \in X$ such that $z'=e_1z+e_2u$, where $e_2u$ lies in $\mathcal {NC}_X$. Since $f$ is a $\mathbb D$-bounded linear 2-functional on $\mathcal M \times [z]$, we define a $\mathbb D$-bounded linear 2-functional $f'$ on $\mathcal M \times [z']$ in such a way that $f'(x,\alpha z)=f(x,\alpha z)$. Then by using Case 1, we get a $\mathbb D$-bounded linear 2-functional $F$ on $X \times [z']$ such that $ \|F\|_\mathbb D = \|f'\|_\mathbb D\; \text{and}\;\;$ $ F(x,\alpha z')= f'(x,\alpha z'), \;\;\text{for\;all}\;\;\; (x,\alpha z')\in \mathcal M \times [z']$. Clearly, $F$ is also a $\mathbb D$-bounded linear 2-functional on $X \times [z]$ such that 
$F(x,\alpha z)= f(x,\alpha z)$, whenever $(x,\alpha z)\in \mathcal M \times [z]$.\\
Further, if $z=0$, then $f(x,\alpha z)=0$, for every $(x, \alpha z) \in \mathcal M \times [z]$. Define $F$ on $X \times [z]$ as $F(x,\alpha z)=0$, for every $(x, \alpha z) \in X \times [z]$ and the theorem follows.                                                                               
\end{proof}                                                                                                               
\begin{corollary} Let  $(X, \|.,.\|_\mathbb D)$ be a 2-normed $\mathbb D$-module and let $x_0$ and $y_0$ be linearly independnt elements of $X$ such that $x_0$ and $y_0$ are not zero divisors in $X$. Then there exists a $\mathbb D$-bounded linear 2-functional $f$ on $X \times [y_0]$ such that $\|f\|_\mathbb D=1$ and $f(x_0,y_0)=\|x_0,y_0\|_\mathbb D$.
\end{corollary}
\begin{proof} Cosider the $\mathbb D$-submodule $[x_0]\times [y_0]$ of $X \times [y_0]$ and define $f_0$ on $[x_0]\times [y_0]$ with $\alpha, \beta \in \mathbb D$  as 
$$ f_0(\alpha x_0, \beta y_0)= \alpha \beta \|x_0,y_0\|_\mathbb D.$$ Clearly $f_0$ is a $\mathbb D$ linear 2-functional with the property that $f(x_0,y_0)=\|x_0,y_0\|_\mathbb D$. 
Further, for any $\alpha, \beta \in \mathbb D$, write $\alpha=e_1\alpha_1+e_2\alpha_2$ and $\beta=e_1 \beta_1+e_2 \beta_2$, where each $\alpha_1, \;\alpha_2, \;\beta_1$ and $\beta_2$ can be zero or non zero real enteries. Now the following cases arises:\\
(i) If $\alpha_2=0$ and $\beta_1=0$, then $|f_0(\alpha x_0, \beta y_0)|_\textbf{k}$= $|f_0(e_1\alpha_1 x_0, e_2 \beta_2 y_0)|_\textbf{k} = 0$.\\ 
(ii) If $\alpha_2=0$ and $\beta_2=0$, then $|f_0(\alpha x_0, \beta y_0)|_\textbf{k}$=$|f_0(e_1\alpha_1 x_0, e_1 \beta_1 y_0)|_\textbf{k}$ \\
$ ~~~~~=e_1 \alpha_1 \beta_1 \|x_0,y_0\|_\mathbb D$= $\|e_1\alpha_1 x_0, e_1 \beta_1 y_0\|_\mathbb D$=$\|\alpha x_0, \beta y_0\|_\mathbb D$.\\
(iii) If $\alpha_1 \neq 0$, $\alpha_2 \neq 0$, $\beta_1 \neq 0$ and $\beta_2 \neq 0$, then $|f_0(\alpha x_0, \beta y_0)|_\textbf{k}=\alpha \beta \|x_0,y_0\|_\mathbb D$ $~~~~~~~~~~~~~=\|\alpha x_0, \beta y_0\|_\mathbb D$.\\
 Similarly we have the other cases. Thus, in all the cases we have $$|f_0(\alpha x_0, \beta y_0)|_\textbf{k} \leq ' \|\alpha x_0, \beta y_0\|_\mathbb D.$$ Hence $f_0$ is a $\mathbb D$-bounded linear 2-functional. Further, (\ref{llp}) implies  $\|f_0\|_\mathbb D=1$. Thus, by using Theorem \ref{2 hahn}, $f_0$ has a 2-linear extension $f$ from $[x_0]\times [y_0]$ to $X \times [y_0]$ such that $f(x_0,y_0)= f_0(x_0,y_0)=\|x_0,y_0\|_\mathbb D$ and have the same norm as $f_0$, that is $\|f\|_\mathbb D=1$.

\end{proof}

\end{section}

\bibliographystyle{amsplain}

\noindent Kulbir Singh,\; \textit{Department of Mathematics,\; University of Jammu, \;Jammu,  J\&K - 180 006, India.}\\
E-mail :\textit{ singhkulbir03@yahoo.com}\\

\noindent Romesh Kumar, \textit{Department of Mathematics, University of Jammu, Jammu, J\&K - 180 006, India.}\\
E-mail :\textit{ romesh\_jammu@yahoo.com}\\

\end{document}